\documentclass[11pt]{amsart}
\usepackage{amssymb}
\usepackage{bbm}
\usepackage{subeqnarray}

\usepackage{url}
\usepackage{}
\newcommand{\R}{\mathbb{R}}

\newcommand{\C}{\mathbb{C}}
\usepackage{bbm}
\usepackage{hyperref}
\usepackage{color}
\newtheorem{theorem}{Theorem}[section]
\newtheorem{lemma}[theorem]{Lemma}
\newtheorem{corollary}[theorem]{Corollary}
\newtheorem{proposition}[theorem]{Proposition}

\theoremstyle{definition}
\newtheorem{definition}[theorem]{Definition}
\newtheorem{example}[theorem]{Example}
\newtheorem{question}[theorem]{Question}

\theoremstyle{remark}

\begin{document}
\title[New Method of Smooth Extension]{New Method of Smooth Extension of Local Maps on Linear Topological Spaces. Applications and Examples.}
\author{Genrich Belitskii}\address{Bengurion University of the Negev, Israel}
\email{Genrich@math.bgu.ac.il}

\author{Victoria Rayskin} \thanks{The second author wants to thank ARO for the grant \# W911NF-19-1-0399, which supported the revision of the techniques discussed in this work, and the manuscript preparation.}
\address{
Tufts University, USA}
\email{Victoria.Rayskin@Tufts.edu}

 \keywords{bump functions, local maps, map extinctions}
         
\begin{abstract}
The question of extension of locally defined maps to the entire space arises in many problems of analysis (e.g., local linearization of functional equations). 
A known classical method of extension of smooth local maps on Banach spaces uses smooth bump functions. However, such functions are absent in the majority of infinite-dimensional spaces. 
We suggest a new approach to localization of Banach spaces with the help of locally identical maps, which we call blid maps. In addition to smooth spaces, blid maps also allow to extend local maps on non-smooth spaces (e.g., $C^q [0, 1]$, $q=0, 1, 2,...$). 

For the spaces possessing blid maps, we show how to reconstruct a map from its derivatives at a point (see the Borel Lemma). We also demonstrate how blid maps assist in finding global solutions of cohomological equations having linear transformation of the argument. We present application of blid maps to local differentiable linearization of maps on Banach spaces.

We discuss differentiable localization for metric spaces (e.g.,  $C^{\infty}(\R)$), prove an extension result for locally defined maps and present examples of such  extensions for the specific metric spaces.
In conclusion, we formulate open problems.
\end{abstract}
\maketitle
\section{Introduction}\label{sec-intro}

The subject of localization of maps goes back to the works of S. Sobolev (\cite{S}) on generalized functions and of K. O. Friedrichs and D. A. Flanders on molifiers. Nowadays, the most frequently used analogous notions are the bump functions, which are smooth, equal to $1$ in some neighborhood and vanish outside of some bigger neighborhood.

There are many examples, where bump functions are used for the study of local properties of dynamical systems in $\R^n$. For instance, see \cite{N} and \cite{St}. J. Palis in his work \cite{P} considers bump functions in Banach spaces. He proves the existence of Lipschitz-continuous extensions of local maps with the help of Lipschitz-continuous bump functions. However, Z. Nitecki (\cite{N}) points out that it is unknown whether the smoothness of these extensions may be higher than Lipschitz.

Even though continuous bump functions exist in all Banach spaces, the majority of infinite-dimensional spaces do not have smooth bump functions. This is an obstacle in the local analysis of dynamical systems in infinite-dimensional spaces. Following V.Z. Meshkov (\cite{M}) we adopt the following
\begin{definition}[$C^q$-smooth spaces]
A space is called $C^q$-smooth, if it possesses a $C^q$-bump function.
\end{definition}

Consider $X=l_p$. If $p=2n$, then $h(||x||^p)$ is a $C^{\infty}$-bump function at $0$, where $h$ is a bump function on $\R$. However, it is known that $l_1$ space does not have $C^1$ bump functions (e.g., \cite{M}). Consequently, $C[0, 1]$ does not have smooth bump functions (this follows from Banach-Mazur Embedding  Theorem). 

In order to allow smooth localization of Banach spaces, we define analogs of bump functions, which we call blid maps (Section~\ref{section-banach}). $C^q$-smooth blid maps exist not only on all $C^q$-smooth spaces, but also on some Banach spaces, which are not $C^q$-smooth. We present specific examples of blid maps for such spaces.

The general topological spaces, such as $C^{\infty}(\R)$ and $C^{\infty}([0,1])$, are frequently discussed in a content of partial differential equations. For this reason, we also discuss how to apply our ideas to the linear topological spaces. In Section~\ref{section-topological} we define blid-differentiable property for topological spaces, present examples of spaces with such property and prove a theorem which asserts existence of global differentiable extension of locally defined maps.

We also discuss applications (Section~\ref{section-applications}) of the localization of the spaces to the questions of solvability of smooth cohomological eqations and to the local differentiable linearization on Banach spaces. The proofs of these results are based on the Borel Lemma for Banach spaces, which can be found in the same section. 

We conclude our paper  (Section~\ref{section-open-q})  with open questions regarding the existence of smooth blid-maps for some non-smooth spaces, validity of Whitney Extension Theorem for non-smooth spaces and existence of Banach spaces without differentiable blid maps.

\section{Banach Spaces}\label{section-banach}
First, let $X$ be a real Banach space, and $Y$ be a real or complex Banach space. We will discuss smooth  local maps $f: X\to Y$ and a possibility of smooth extension of the maps. Since Banach spaces are equipped with norms, we can consider Fr\'echet derivatives. In the topics related to Banach spaces, we will assume that differentiation is defined in Fr\'echet sense. 

The map's extension is usually not unique and can be studied in the context of the equivalence class of $f$, i.e. a germ $[f]$.  Recall that a germ $[f]$ at $x\in X$  is the equivalence class of local maps, such that any pair of the class members coincides on some neighborhood of $x$. Each element of the class is called a representative of a germ. Occasionally, we denote germ $[f]$ as $f$. In the future, without loss of generality, we will assume that $x=0$.
We are interested in the question of existence of a global representative of the germ.

Consider a $C^q$ germ. Does there exist a $C^q$ global representative of the germ? Suppose there exists a representative with $q$ bounded derivatives. Does there exist a global representative which also has $q$ bounded derivatives? To answer these question, we introduce special maps, discussed below.

\begin{definition}[$C^q$-smooth blid maps]\label{def-blid} A $C^q$ blid map for a Banach space $X$ is a global {\bf B}ounded {\bf L}ocal {\bf Id}entity at zero $C^q$ map $H : X \to X$.
\end{definition}

The idea of extensions with the help of blids first appears in \cite{B}, later in \cite{BR}. The Definition~\ref{def-blid} was introduces in~\cite{BR1} and was motivated with the following example.

\begin{example}{
 The $C^{\infty}$ germ, defined in the neighborhood of $0\in C[0,1]$ 
$$
f(x)=\int_0^1 \frac{dt}{1-x(t)}
$$
has a global $C^{\infty}$ representative:
$$
\int_0^1 \frac{dt}{1-h(x(t))x(t)}.
$$
Here blid map $h(s)s$ is defined with the help of the bump function $h$, such that $h(s)=1$ on $|s|<1/3$ and $0$ on $|s|>1/2.$
}
\end{example}

In \cite{BR1}, we generalize the idea of smooth extension of a locally defined map via composition of the map with a smooth blid-map. This method allows us to prove the Borel Lemma \ref{thm-Borel} for Banach spaces. Many questions of local dynamics can be addressed with the help of this theorem.

\begin{theorem}\label{thm-extension-Banach} Let a space $X$ possesses a $C^q$-blid map $H$. Then for every Banach space $Y$ and any $C^q$-germ $f$ at zero from $X$ to $Y$ there exists a global $C^q$-representative. Moreover, if all derivatives of $H$ are bounded, and $f$ contains a local representative bounded together with all its derivatives, then it has a global one with the same property.
\end{theorem}

 Obviously, if a space is $C^q$-smooth, it possesses  $C^q$-blid map. However, there are examples of Banach spaces that have blid-maps, but do not have bump functions of the same smoothness. We will illustrate this idea with the following examples (for details and proofs see \cite{BR1}) of blid-maps in various Banach spaces:

\begin{example} Suppose $X$ has a $C^q$ bump function $h:X\to \R$. Then, $H(x) =h(x)x$ is a $C^{q}(X)$ blid map.
\end{example}

\begin{example} Let $X=C[0,1]$ and $h$ be a smooth bump function on $\R$. Then, $H(x)(t) =h(x(t))x(t)$ is a $C^{\infty}(X)$ blid map.
\end{example}
\begin{example} More generally, suppose $X=C(M)$ where $M$ is a compact space. Then, $H(x)(t) =h(x(t))x(t)$ is a $C^{\infty}(X)$-blid.
\end{example}
\begin{example} Let $X=C^q[0,1]$. Then, a $C^{\infty}(X)$-blid $H(x)(t)$ can be defined as
 $$H(x)(t)=\sum_{j=0}^{q-1}\frac{t^j}{j!}h(x^{(j)}(0))x^{(j)}(0) + \int_0^t\,dt_1 \int_0^{t_1}\,dt_2... \int_0^{t_{q-1}}h\left( x^{(q)}(s) \right)x^{(q)}(s) \,ds .$$
\end{example}
There are also some examples of subspaces, where blid maps can be constructed:

\begin{example} Suppose $X$ possess a $C^q$-blid map $H$, and a subspace $X_1$ of $X$ be $H$-invariant. Then, the restriction $H_1 = H | X_1$ is a $C^q$-blid map on $X_1$.
\end{example}
\begin{example} Assume $\pi : X \to X$ is a bounded projector and $X$ possess $C^q$-blid map $H$. Then, the restriction $\pi(H) \vert Im(\pi)$ is a $C^q$-blid map on $Im(\pi)$, while the restriction $(H-\pi (H))\vert Ker(\pi)$ is a $C^q$-blid map on $Ker(\pi)$. Consequently,
if $X_1 \subset X$ is a subspace, such that there exists another subspace of $X$, so
that these two form a complementary pair, then $X_1$ possesses a blid map.
\end{example}


\section{Linear Topological Spaces}\label{section-topological}
As we noticed in section~\ref{sec-intro} localization on topological linear spaces (e.g.,  $C^{\infty}(D)$, where $D$ is some smooth manifold) is important for the study of partial differential equations.

It is not possible to define Fr\'echet differentiability on a general linear topological space. For this reason,  we will use weaker definitions of differentiation. 
As we have seen in Section~\ref{section-banach}, an extension of maps with the help of blids requires composition. Thus, we will discuss differentiability that satisfies the Chain Rule (in particular,  we cannot use G\^ateaux derivative). We will work with the bounded-differentiability, a stronger, compact (Hadamard) differentiability, and (if it can be defined) the strongest here, Fr\'echet differentiability. Let us recall these definitions.

\begin{definition}[Bounded differentiability]
 The map $f :X \to Y$ is bounded-differentiable at $x\in X$, if for every bounded subset $S\subset X$ and every $h\in S$ and $t\in \R$ 
$$
f(x+th) - f(x) = tA h +r(th)
$$
with 
$$r(th)/t \to 0$$ uniformly in $h$ as $t\to 0$.
\end{definition}

\begin{definition}[Compact (Hadamard) differentiability]
The map $f:X \to Y$ is compact (Hadamard) differentiable at $x\in X$, if
     $$ f(x+t_n h_n) -f(x)=t_n A h +o(t_n)$$
as $t_n \to 0$, and $h_n \to h$.
\end{definition}

If both $X$ and $Y$ are Banach spaces with the norms $||.||_1$ and $||.||_2$ respectively, then Fr\'echet differentiation is well defined. 
\begin{definition}[Fr\'echet differentiability]
The map $f$ is Fr\'echet differentiable at $0$ if
       $$ \lim_{h\to 0}||r(h)||_2/||h||_1 =0.$$
\end{definition}

These definitions define the same derivative if it exists, and differ only by the definition of the remainder term.

\begin{definition}
A space $X$ satisfies blid-differentiable property if for every neighborhood $U\subset X$ of $0$ there is a differentiable map $H$ defined on $X$, locally coinciding with the identity map, such that $H(X)\subset U$.
\end{definition}
Let us recall that a neighborhoods base of zero is a system $B=\{V_\alpha\}$ of neighborhoods of 0, such that for any neighborhood $U\subset X$ of 0 there exists some $V_\beta \in
 B$, $V_\beta \subset U$.

     Therefore, if there is a neighborhoods base $B$ such that for every $V_\alpha$ from $B$ there exists local identity $H_\alpha$, $H_\alpha(X) \subset  V_\alpha$, then $X$ satisfies the blid-property.
\begin{proposition}[Extension of Local Maps]\label{thm-main}
If $X$ satisfies blid-differentiable property, then every differentiable germ $f:X\to Y$ has a global differentiable representative.
\end{proposition}
\begin{proof}
 Let $f$ be a local representative of the germ defined on a neighborhood $U\subset X$ of zero. Let $H:X \to X$ be a differentiable local identity map such that

$H(X)\subset U$. Then the  map 
$$
                   F(x)=f(H(x)), \ \ x\in X
$$
is a global representative of the germ as we need. 
\qed\end{proof}
Let $X$ be a metric space with a metric $d(x,y)$.  Here we consider germs of maps from $X$ into an arbitrary linear  topological space $Y$.  Although instead of Fr\'echet differentiation (which is not defined for general metric spaces) we use bounded and compact (Hadamar) differentiation. The neighborhoods base $B$ can be chosen as a collection $\{B_c\}_c=\{x\in X: d(x,0)<c\}_c$. Then the space $X$ satisfies differentiable-blid property if for every $c$ there exists a differentiable, local identity map $H_c:X\to X$ such that $d(H_c(x),0)<c$ for all $x$, i.e., $H_c(X)\subset B_c$. 

In particular, if topology on $X$ is defined by countable collection of norms $||x||_{k}$, then the metric can be written as
$$
d(x,y):= \sum_{k=0}^\infty{\frac{1}{2^k}\cdot \frac{||x-y||_k}{||x-y||_k+1} }.
$$

\begin{proposition} Suppose for every $k=0,1,...$ there exists a global differentiable local identity map $\mathcal H_k$ such that
$$
            \sup_x ||\mathcal{H}_k(x)||_k< \infty.
$$
Then $X$ satisfies the differentiable blid property.
\end{proposition}
\begin{proof} For a given $c>0$ choose any 
\begin{eqnarray}\label{k-c-inequality}
k>1- \ln c/\ln2
\end{eqnarray}
and let $\mathcal H_k$ be such that
$$
                      ||\mathcal H_k(x)||_k <N,\ x\in X.
$$
Set 
$$
             H_c(x)=\frac{c}{4N}\mathcal H_k\left(\frac{4N}{c} x\right).
$$
Then inequality~\ref{k-c-inequality} and the fact that $||x||_j$ is monotonically increasing with $j$ imply that   
$$
                d(H_c(x),0)<c,
$$
i.e. $H_c(X)\in B_c$.
\qed\end{proof}

In the following subsections, we present the examples of the spaces with differentiable blid property and state the existence of extension of locally defined maps on these spaces.
\subsection{The Space of Smooth Functions on the Real Line}
The space $X=C^q(\R)$ ($0\leq q <\infty$) of all smooth functions on $\R$ is endowed with the collection of norms
$$
               ||x||_k=\max_{t\in[-k,k]}\max_{l\leq q}|x^{(l)}(t)|.
$$
\begin{lemma}
The space $X$ possesses the bounded- (consequently compact-)  differentiable blid property.
\end{lemma}
\begin{proof}
Let $h(u)$ be a $C^{\infty}$-bump function on $\R$. Note, $a=\sup_{u\in\R}h(u)u<\infty$. Then 
$$
                H(x)(t)=\left\{
\begin{array}{l}
h(x(t))x(t),\  q=0\\
\sum_{j=0}^{q-1}\frac{t^j}{j!}h(x^{(j)}(0))x^{(j)}(0) + \int_0^t\,dt_1 \int_0^{t_1}\,dt_2 ... \int_0^{t_{q-1}} h\left( x^{(q)}(s) \right)x^{(q)}(s) \,ds,\  q\geq 1
\end{array}
\right.
$$
is differentiable local identity map, and
$$
                      ||H(x)||_k< ae^k,\  k=0,1,..., \ x\in X.
$$
\qed\end{proof}
\begin{corollary} Every bounded- (consequently compact-) differentiable germ at $0\in C^q(\R)$ has a global  differentiable (in the corresponding sense) representative.
\end{corollary}
\subsection{The Space of Infinitely Differentiable Functions on a Closed Interval}

The space $X=C^{\infty}[0,1]$ is endowed with the collection of norms 
$$
 ||x||_k= \max_{j\leq k}\max_{t\in[0,1]}|x^{(j)}(t)|.
$$
\begin{lemma}
The space $X$ possesses the bounded- (consequently compact-) differentiable property.
\end{lemma}
\begin{proof}
Let $h(u)$ be the same bump function on $\R$ as above. Then
$$
        H_{0}(x)(t)=h(x(t))x(t)
$$
is differentiable local identity map, and
$$
           ||H_{0}(x)||_0<a.
$$

   Further, let $k>0$. Then 
$$
H_{k}(x)(t)=\sum_{j=0}^{k-1}\frac{t^j}{j!}h(x^{(j)}(0))x^{(j)}(0) + \int_0^t\,dt_1 \int_0^{t_1}\,dt_2 ... \int_0^{t_{k-1}} h\left( x^{(k)}(s) \right)x^{(k)}(s) \,ds .
$$
is differentiable local identity map, and
$$
                   ||H_{k}(x)||_k< ae^k, \ \ k=0,1,...,\ \ x\in X.
$$
\qed\end{proof}
\begin{corollary}
Every bounded- (consequently compact-) differentiable germ at $0\in C^{\infty}[0,1]$ has a global representative. 
\end{corollary}

\subsection{The Space of Infinitely Differentiable Functions on the Real Line}

The space $X=C^{\infty}(\R)$ is endowed with the collection of norms 
$$
 ||x||_{k}= \max_{j\leq k}\max_{t\in[-k,k]}|x^{(j)}(t)|, \ k=0,1,2,...
$$
\begin{lemma}
The space $X$ possesses the bounded- (consequently compact-) differentiable property.
\end{lemma}
\begin{proof}
Let $h (u)$ be the same bump function on $\R$ as above. Then
$$
        H_{0}(x)(t)=h(x(t))x(t)
$$
is differentiable local identity map, and
$$
           ||H_{0}(x)||_0<a.
$$

   Further, let $k>0$. Then 
$$
H_{k}(x)(t)=\sum_{p=0}^{k-1}\frac{t^p}{j!}h(x^{(p)}(0))x^{(p)}(0) + \int_0^t\,dt_1 \int_0^{t_1}\,dt_2 ... \int_0^{t_{k-1}} h\left( x^{(k)}(s) \right)x^{(k)}(s) \,ds .
$$
is differentiable local identity map, and
$$
                  ||H_{k}(x)||_k< ae^k, \ \ k=0,1,...,\ \ x\in X.
$$
\qed\end{proof}
\begin{corollary}
Every bounded- (consequently compact-) differentiable germ at $0\in C^{\infty}(\R)$ has a global representative. 
\end{corollary}

\section{Applications}\label{section-applications}
Frequently, in the questions of local analysis and local dynamical systems bump functions are used. Since blid maps substitute the bump functions, they allow localization of a broader class of spaces. First, the blid maps were used in \cite{B} for smooth conjugation of two $C^{\infty}$ diffeomorphisms on some Banach spaces. In the later works \cite{BR} and \cite{R} we discuss the conditions when two $C^{\infty}$ diffeomorphisms on some Banach spaces are locally $C^{\infty}$-conjugate. Below, we discuss applications of blid maps to the differentiable linearization (without non-resonance assumption), and applications to the cohomological equations. For the proofs of these results we need the Borel Lemma for Banach spaces. With the help of blid-maps we are able to prove these results for Banach spaces. 

\subsection{The Borel Lemma}
In this section we state the Borel Lemma proved in \cite{BR1}. For finite dimensional $X$, the Borel Lemma is a particular case of the celebrated Whitney theorem on the extension from a closed set. The use of blid-maps in our proofs is analogous to the use of the bump functions in the proofs of finite-dimensional case, but infinite dimensional version of the proof requires some estimates on the growth of the derivatives of the blid-maps.
\begin{theorem}[The Borel lemma]\label{thm-Borel} Let a Banach space X possess a $C^\infty$-blid map with bounded derivatives of all orders.
Then for any Banach space Y and any sequence $\{P_j\}_{j=0}^{\infty}$ of continuous homogeneous polynomial maps from $X$ to $Y$ there is a
$C^\infty$-map $f: X\to Y$ with bounded derivatives of all orders such that   $P_j(x)=f^{(j)}(0)(x)^j$ is satisfied for all $j=0,1,...$
 \end{theorem}

\subsection{Cohomological Equations}\label{section-cohomol}
In this section we outline the main ideas of the application of the blid maps to the solutions of cohomological equations. For the detailed discussion please see \cite{BR1}.
Given a map $F : X \to X$, ($X$ is a Banach space) we want to find a $C^{\infty}$ $g: X\to \C$, that satisfies the following cohomological equation:
\begin{equation}\label{coh-eqn}
g(Fx) - g(x) = f(x)
\end{equation} 
For a broad overview of various versions of the equation see the works of Yu.I. Lyubich (e.g., \cite{L}). Also, for a discussion of smooth cohomological equations we recommend the book \cite{B-T}. 
\\
In our example, we will assume that $F$ is linear and denote it by $A$.

Define a homogenious polinomial map $P_n(x)=f^{(n)}(0)(x)^n$. We will search for a homogeneous, degree $n$, polynomial solutions $Q_n(x)$ ($n=1,2...$) such that
\begin{eqnarray*}
\hspace{1in}{(L_n - id )Q_n(x) = P_n(x),\  \  n=1,2,3....,} \hspace{1.2in} (n)
\end{eqnarray*}
where $L_nQ_n (x)= (Q_n(Ax))^{(n)}$. 

If for every $n$ equation~($n$) is solvable, we call the cohomological equation~(\ref{coh-eqn}) formally solvable. Then we can use Borel Lemma to reduce the equation~(\ref{coh-eqn})  to the equation in flat functions (with $0$ Taylor coefficients).
 Then, applying some decomposition results (see~\cite{BR1}) for the space $X$, we can formulate the conditions for the solvability of the original cohomological equation: 
\begin{theorem} Let $A$  be a hyperbolic linear automorphism, and $X$ possesses a $C^\infty$-blid map with bounded derivatives on $X$. If all derivatives of f are bounded on every bounded subset, and cohomological eqn. is formally solvable at zero (i.e. each $n$-th equation has continuous solution, $n=1,2,...$), then there exists a global $C^\infty$-solution $g(x)$.
 \end{theorem}


\subsection{Differentiable Linearization without Non-resonance Assumption}\label{section-lin}
Local linearization and normal forms are convenient simplification of complex dynamics. In this section we discuss differentiable linearization on Banach spaces.
For a diffeomorphism $F$ with a fixed point $0$, we would like to find a smooth transformation $\Phi$ defined in a neighborhood of $0$ such that $\Phi\circ F \circ \Phi^{-1}$ has a simplified (polynomial) form called the normal form. If $\Phi\circ F \circ \Phi^{-1}=DF=\Lambda $ is linear, the conjugation is called linearization. There are two major questions in this area of research: how to increase smoothness of the conjugation $\Phi$, and whether it is sufficient to assume low smoothness of the diffeomorphism $F$.

Hartman and Grobman independently showed that if $\Lambda$ is hyperbolic, then for a diffeomorphism $F$ there exists a local homeomorphism $\Phi$ such that $\Phi \circ F \circ \Phi^{-1}=\Lambda$. Different proofs were given by Pugh in \cite{Pu}. A higher regularity of $\Phi$ has been an active area of research (see, for example, \cite{P}, \cite{vS}, \cite{GHR}, \cite{ZLZ}).

The first attempt to answer the question of differentiability of $\Phi$ at the fixed point $0$ under hyperbolicity assumption was made in \cite{vS}, but an error was found and discussed in \cite{R1}. Later, in \cite{GHR}, Guysinsky, Hasselblatt and Rayskin presented correct proof. However, it was restricted to $F\in C^{\infty}$ (or more precisely, it was restricted to $F\in C^k$, where $k$ is defined by complicated expression). It was conjectured in the paper that the result is correct for $F\in C^2$, as it was announced in \cite{vS}.

Zhang, Lu and Zhang, in their Theorem 7.1 published in~\cite{ZLZ} showed that for a Banach space diffeomorphism $F$ with a hyperbolic fixed point and $\alpha$-H{\"o}lder $DF$, the local conjugating homeomorphism $\Phi$
is differentiable at the fixed point. Moreover, 
$$
\Phi(x) = x+O(||x||^{1+\beta}) \mbox{\  and \ } \Phi^{-1}(x) = x+O(||x||^{1+\beta})
$$
as $x\to 0$, for certain $\beta \in (0,\alpha]$. 

There are two additional assumptions in this theorem. The first one is the spectral band width inequality. The authors explain that this inequality is sharp if the spectrum has at most one connected component inside of the unit circle in $X$, and at most one connected component outside of the unit circle in $X$. For the precise formulation of the spectral band width condition we refer the reader to the paper \cite{ZLZ}.
It is important (and it is pointed out in \cite{ZLZ}) that this is not a non-resonance condition. The latter is required for generic linearization of higher smoothness. 

The second assumption is the assumption that the Banach space must possess smooth bump functions. 
It is conjectured in the paper that the second assumption is a necessary condition. 

In this section we explain that this conjecture is not correct (see Theorem~\ref{thm-diff}). The bump function condition can be replaced with the less restrictive blid map condition. 
Blid maps allow to reformulate Theorem 7.1 in the following way:
\begin{theorem}\label{thm-diff}
Let $X$ be a Banach space possessing a differentiable blid map with bounded derivative.
Suppose $F:X\to X$ is a diffeomorphism with a hyperbolic fixed point,  $DF$ is $\alpha$-H{\"o}lder, and the spectral band width condition is satisfied.
Then, there exists local linearizing homeomorphism $\Phi$ which 
is differentiable at the fixed point. Moreover, 
$$
\Phi(x) = x+O(||x||^{1+\beta}) \mbox{\  and \ } \Phi^{-1}(x) = x+O(||x||^{1+\beta})
$$
as $x\to 0$, for certain $\beta \in (0,\alpha]$. 
\end{theorem}
In particular, we have the following 
\begin{corollary}
Let $X=C^q[0,1]$.
Suppose $F:X\to X$ is a diffeomorphism with a hyperbolic fixed point,  $DF$ is $\alpha$-H{\"o}lder, and the spectral band width condition is satisfied.
Then, the local conjugating homeomorphism $\Phi$
is differentiable at the fixed point. Moreover, 
$$
\Phi(x) = x+O(||x||^{1+\beta}) \mbox{\  and \ } \Phi^{-1}(x) = x+O(||x||^{1+\beta})
$$
as $x\to 0$, for certain $\beta \in (0,\alpha]$. 
\end{corollary}

Below we justify Theorem~\ref{thm-diff}
\begin{proof}
Zhang, Lu and Zhang showed that for the conclusion of their Theorem 7.1 it is enough to satisfy the inequalities 1 and 2 (see  \ref{ineq} below),  which are called condition (7.6) in their paper.

In order to apply the blid maps instead of bump functions to the inequalities \ref{ineq}, it is sufficient to construct a bounded blid map, which has only first-order bounded derivative. I.e., let blid map $H(x): X\to X$ be as follows:
\begin{eqnarray}
\begin{array}{l}
\mbox{1. \ }  H(x) = x \mbox{\ for\ } ||x||<1\\
\mbox{2. \ } H\in C^1 \mbox{\ and\ } ||H^{(j)}(x)||\leq c_j,  \ j=0,1 .
\end{array}
\end{eqnarray} 

The condition (7.6) of \cite{ZLZ} is:
\begin{eqnarray}\label{ineq}
\begin{array}{l}
\mbox{1. \ } \sup_{x\in X}||DF(x) -\Lambda|| \leq \delta_{\eta}\\
\mbox{2. \ } \sup_{x\in V\setminus O}\left\{||DF(x) - \Lambda|| / ||x||^{\alpha}\right\} = M < \infty 
\end{array}
\end{eqnarray} 

Let $DF -\Lambda = f$. Define
$$
\tilde{f}(x):= f\left( \delta H(x/ \delta) \right)
$$

We will show that if $f$ satisfies (7.6), then so does $\tilde{f}$.
$$
\sup_{x\in X}|| D\tilde{f}(x)|| \leq \sup_{x\in X} || D f(x) || \cdot \sup_{x\in X} || DH(x) || \leq \delta_{\eta} \cdot c_1.
$$
Thus, the first inequality of (7.6) holds for $\tilde{f}$. For the second inequality we have the following estimate:
$$
\frac{|| D \tilde{f}(x)||}{||x||^{\alpha}} \leq \frac{|| Df \left(\delta H(x/\delta)\right) ||}{|| \delta H(x/\delta) ||^{\alpha}} \cdot \left( \frac{|| \delta H(x/\delta) ||}{||x||}  \right)^{\alpha}.
$$
The second multiple is bounded, because for small $x$ (say, $||x/\delta||<\epsilon$ for some $\epsilon>0$) we have
$$
\frac{|| \delta H(x/\delta) ||}{||x||} < c_1 + o(1),
$$
while for $||x/\delta|| \geq \epsilon$
$$
\frac{|| \delta H(x/\delta) ||}{||x||} < c_0/\epsilon.
$$
I.e., $\frac{|| \delta H(x/\delta) ||}{||x||}$ is less than some constant $m$.
Then, 
$$
\sup_{x\in V\setminus O}\frac{|| D \tilde{f}(x)||}{||x||^{\alpha}} \leq
\sup_{0<||x||<\delta c_0}\left\{||D f(x)|| / ||x||^{\alpha}\right\} \cdot \sup_{x\in X}|| D H(x)|| \cdot m^{\alpha}
$$
$$=\sup_{0<||x||<\delta c_0}\left\{||D f(x)|| / ||x||^{\alpha}\right\}c_1\cdot m^{\alpha}.
$$

This quantity is bounded by $M c_1 m^{\alpha}$ if $\delta$ is sufficiently small.
\qed\end{proof}

\section{More examples and open questions}\label{sec-questions}

One of the important questions of local analysis on Banach spaces is the following. Do Banach spaces without smooth blid maps  exist? Recently, affirmative answer was presented in \cite{DH} (also see \cite{HJ}). The authors proved that there exist Banach spaces that do not allow $C^2$-extension (and hence the $C^2$-blid map).
\begin{question}\label{Q-arbitrary-space} For which spaces do smooth blid maps exist? 
Do they exist on $l_p$, with non-even $p$?
\end{question}

\begin{question}\label{questionC-1} 
Are there Banach spaces without differentiable blid maps? 
\end{question}

In the Theorem~\ref{thm-extension-Banach} we considered a $C^q$-germ at a point. For such germs the existence of a local representative with bounded derivatives implies the existence of the global one with the same properties.

How can we extend germs of maps defined at a closed subset $S\subset X$? For this construction we need to define smooth blid maps at $S$. More precisely, generalizing the definition of germs at a point, we will say that maps $f_1$ and $f_2$ from neighborhoods $U_1$ and $U_2$ of $S$ into $Y$ are {\it equivalent}, if they coincide in a (smaller) neighborhood of $S$.
Every equivalence class is called a {\it germ at $S$}. We pose the same question. Given a $C^q$-germ at $S$, does there exist a global representative? Assume there exist a $C^q$-map $H:X \to X$ whose image $H(X)$ is contained in a neighborhood  $U$ of $S$
and which is equal to the identity map in a smaller neighborhood. Such maps we call {\it smooth blid maps at S}. Then every local map $f$ defined in $U$ can be extended on the whole $X$.  It suffices to set $F(x)=f(H(x))$.

In the next example, we construct the map $H$ for a segment (in particular, for a ball).
\begin{example}\label{example-segment}
Let $S(A)$ be a set of all functions $x\in C[0,1]$ whose graphs $(t,x(t))$ are
contained in a closed $A\subset\R^2$, where $A$ is chosen in such a way that $S(A)\neq \emptyset$. Let  $h(t,x)$ be a $C^\infty$-function, which is equals to $1$ in a neighborhood of $A$ and vanishes 
outside of a bigger set. Then, for an arbitrary $y\in C[0,1]$
$$
             H_y(x)(t)=y(t)+h(t,x(t))(x(t)-y(t))
$$
is a $C^{\infty}$-blid map for $S(A)$.

If $A=\left\{\{t,x\}: \min(\psi(t),\phi(t))\leq x\leq  \max(\psi(t),\phi(t))\right\}$ for some $\phi,\psi \in C[0,1]$, then $S(A)$ can be thought of as a segment $[\phi,\psi]\subset C[0,1]$. 

In particular, given $z\in C[0,1]$ and a constant $r>0$, setting $\phi =z-r$ and $\psi=z+r$, we obtain the ball $B_r(z)=\{ x: ||x-z||\leq r \}\subset C[1,0]$.

Every $C^q$-germ at $[\phi,\psi]\subset C[0,1]$ contains a global representative. 

Note, this example has an obvious generalization to segments and balls in $C^k[0,1]$.
\end{example}

The Question \ref{Q-arbitrary-space} and Example \ref{example-segment} bring us to the next question. 
\begin{question} For which pairs $(S,X)$ do similar constructions exist? In particular, can a smooth blid map be constructed for any bounded subset $S$ of a space $X$ possessing a smooth blid map? For example, we do not know whether a smooth blid map can be constructed for a sphere $S=\{x\in C[0,1]: ||x||=r\}$. 
\end{question}
\begin{question} The Borel lemma for finite-dimensional spaces is a particular case of the well-known Whitney extension theorem from a closed set $S\subset \R^n$. What is an infinite-dimensional version of the Whitney theorem? 
\end{question}

In Section~\ref{section-topological} we presented several examples of linear topological spaces with the differentiable blid property.
\begin{question}
Which linear topological spaces have differentiable blid property?
\end{question}

Linearization is a convenient simplification in the study of local dynamics. In some cases partial differential equations can be studied in terms of operators on linear topological spaces. Thus, there arises the question of differentiable linearization.
\begin{question} Is it possible to generalize the Theorem~\ref{thm-diff} for the case of linear topological spaces (e.g., space of $C^{\infty}$ functions), which posses differentiable blid property. 
\end{question}

\

%
%


\begin{thebibliography}{6}
%
\bibitem[B]{B} G. Belitskii, \textit{The Sternberg theorem for a Banach space,} Funct. Anal. Appl., \textbf{18} (1984), 238--239. MR 0757253 (86b:58097), \url{http://link.springer.com/article/10.1007\%2FBF01086163}
\bibitem[B-T]{B-T} G. Belitskii, V. Tkachenko, One-dimensional functional equations, \textit{Oper. Th.: Adv. and Appl., 144, Birkh{\"a}auser Verlag,} 2003
\bibitem[BR]{BR} G. Belitskii, V. Rayskin, \textit{Equivalence of families of diffeomorphisms on Banach spaces}, Math. preprint archive, UT Austin, 07-71. \url{https://www.ma.utexas.edu/mp_arc-bin/mpa?yn=07-71}
\bibitem[BR1]{BR1} G. Belitskii and V. Rayskin, \textit{A New Method of Extension of Local Maps of Banach Spaces. Applications and Examples}, Contemporary Mathematics, AMS series, (2019), {\bf 733}

\bibitem[BR2]{BR2} G. Belitskii, V. Rayskin,   \emph{Extension of differentiable local mappings on linear topological spaces}, preprint, arXiv: 1812.11064.

\bibitem[DH]{DH} S. D'Alessandro and P. Hajek, \textit{Polynomial algebras and smooth functions in Banach spaces}, Journal of Functional Analysis \textbf{266} (2014), 1627-1646
\bibitem[GHR]{GHR} M. Guysinsky, B. Hasselblatt, V. Rayskin, \textit{Differentiability of the Hartman-Grobman linearization}, Discrete Contin. Dyn. Syst., (2003) 9 {\bf 4} 979 - 984 pp
\bibitem[HJ]{HJ} P. Hajek and M. Johanis, Smooth Analysis in Banach Spaces, \textit{Walter de Gruyter, GmbH, Berlin}, 2014 497 pp
\bibitem[M]{M} V.Z. Meshkov, \textit{Smoothness properties in Banach spaces,} Studia Mathematica, \textbf{63} (1978), 111--123. MR 0511298 (80b:46027), \url{http://matwbn.icm.edu.pl/ksiazki/sm/sm63/sm6319.pdf}
\bibitem[L]{L} Yu. Lyubich, \textit{The cohomological equations in nonsmooth categories}, arXiv:1211.0229v1 [math. FA] 1 Nov.2012, \url{https://arxiv.org/pdf/1211.0229.pdf}
\bibitem[N]{N} Z. Nitecki, Differentiable Dynamics. An Introduction to the Orbit Structure of Diffeomorphisms, \textit{The MIT Press, Camridge, Mass.-London}, 1971. xv+282 pp. MR 0649788 (58 \#31210)
\bibitem[P]{P} J. Palis, \textit{Local Structure of Hyperbolic Fixed Points in Banach Space},
Anais da Academia Brasileira de Ciencias, \textbf{40} (1968), 263--266
\bibitem[Pu]{Pu}, C. C. Pugh, \textit{On a theorem of P. Hartman}, Amer. J. Math. (1969) 91, 363 - 367 pp

\bibitem[R]{R} V. Rayskin, \textit{Theorem of Sternberg-Chen modulo central manifold for Banach spaces,} Ergodic Theory \& Dynamical Systems, \textbf{29} (2009), \textit{no. 6},1965--1978. MR 2563100 (2011a:37038), \url{https://doi.org/10.1017/S0143385708000989}
\bibitem[R1]{R1} V. Rayskin, \textit{$\alpha$-H{\"o}lder linearization}, Journal of Differential Equations (1998) 147 {\bf 2} 271 - 284 pp

\bibitem[S]{S} S. Soboleff \textit{Sur un th{\'e}or{\`e}me d'analyse fonctionnelle}, Rec. Math. [Mat. Sbornik] N.S., \textbf{4(46):3} (1938), 471--497
\bibitem[St]{St} S. Sternberg, \textit{On the structure if local homeomorphisms of Euclidean n-space II}, Amer. J. Math., \textbf{80} (1958), 623--631

\bibitem[vS]{vS} S. van Strien, \textit{Smooth linearization of hyperbolic fixed points without resonance conditions}, Journal of Differential Equations, (1990), 85, {\bf1}, 66 - 90 pp

\bibitem[ZLZ]{ZLZ}
 author  = { W. Zhang and K. Lu and W. Zhang},
 title   = "Dfifferentiability of the Conjugacy in the Harman-Grobman Theorem", 
 journal = {Transactions of the American Mathematical Society} (2017) 369 {\bf 7} 4995 - 5030 pp

\bibitem {smit:wat}
Smith, T.F., Waterman, M.S.: Identification of common molecular subsequences.
J. Mol. Biol. 147, 195?197 (1981). \url{doi:10.1016/0022-2836(81)90087-5}

\bibitem {may:ehr:stein}
May, P., Ehrlich, H.-C., Steinke, T.: ZIB structure prediction pipeline:
composing a complex biological workflow through web services.
In: Nagel, W.E., Walter, W.V., Lehner, W. (eds.) Euro-Par 2006.
LNCS, vol. 4128, pp. 1148?1158. Springer, Heidelberg (2006).
\url{doi:10.1007/11823285_121}

\bibitem {fost:kes}
Foster, I., Kesselman, C.: The Grid: Blueprint for a New Computing Infrastructure.
Morgan Kaufmann, San Francisco (1999)

\bibitem {czaj:fitz}
Czajkowski, K., Fitzgerald, S., Foster, I., Kesselman, C.: Grid information services
for distributed resource sharing. In: 10th IEEE International Symposium
on High Performance Distributed Computing, pp. 181?184. IEEE Press, New York (2001).
\url{doi: 10.1109/HPDC.2001.945188}

\bibitem {fo:kes:nic:tue}
Foster, I., Kesselman, C., Nick, J., Tuecke, S.: The physiology of the grid: an open grid services architecture for distributed systems integration. Technical report, Global Grid
Forum (2002)

\bibitem {onlyurl}
National Center for Biotechnology Information. \url{http://www.ncbi.nlm.nih.gov}


\end{thebibliography}
\end{document}